\documentclass[]{article}
\usepackage{graphicx}
\usepackage{graphics,color}
\usepackage{dsfont}
\usepackage{amssymb,amsmath,amsthm,epsfig}

\newtheorem{theorem}{Theorem}
\newtheorem{lemma}{Lemma}

\newtheorem{proposition}{Proposition}

\newtheorem{remark}{Remark}
\newcommand{\norm}[1]{{\left\|{#1}\right\|}}

\newcommand{\fact}[1]{#1\mathpunct{}!}

\newcounter{reh}
\newcounter{rek}
\begin{document}

\begin{center}
{\large {\bf Reproducing kernels based schemes for nonparametric regression}}\\
\vskip 1cm Bilel Bousselmi$^a$, Jean-Fran\c{c}ois Dupuy $^a$   and Abderrazek Karoui$^b$ \footnote[1]{Corresponding author,
Email address: Abderrazek.Karoui@fsb.rnu.tn\\
This work was supported in part by the  
 DGRST  research grant  UR 13ES47 and the PHC-Utique research project 20G1503.}

\vskip 0.5cm {\small
\noindent
$^a$ University of Rennes, INSA Rennes, CNRS, IRMAR - UMR 6625, F-35000 Rennes, France.\\
\noindent $^b$ University of Carthage,
Department of Mathematics, Faculty of Sciences of Bizerte, Jarzouna 7021,  Tunisia.}\\

\end{center}

\noindent{\bf Abstract}--- In  this work, we develop and study an empirical projection operator  scheme for solving  nonparametric regression problems. This scheme is based on an  approximate projection of the regression function over a suitable reproducing kernel Hilbert space (RKHS). The RKHS considered in this paper are generated by the Mercer kernels given by  the Legendre Christoffel-Darboux and convolution Sinc kernels. We provide error and convergence analysis of the proposed scheme under the assumption that the regression function belongs to some suitable functional spaces. We also consider the popular  RKHS regularized least square minimization  for nonparametric regression. In particular, we check the numerical stability of this second scheme and we provide its convergence rate in the special case of the Sinc kernel. Finally, we illustrate the proposed methods by various numerical simulations.\\


\section{Introduction} 

Given a complete metric space $\mathcal X$ and an output space $\mathcal Y,$ one main issue of learning theory is to develop algorithms that take a training set $\{(X_i, Y_i), 1\leq i\leq n\}$ in $\mathcal X\times \mathcal Y$ and return a function $f$ such that for $x\in \mathcal X,$ $f(x)$ is a good estimate (or prediction) of the corresponding output $y:=y(x)$. Observations $(X_i,Y_i)$ are assumed to be drawn from a joint probability measure $\rho$ on $\mathcal X \times \mathcal Y.$ In the special case where $\mathcal Y$ is a measurable  subset of $\mathbb R,$ this learning problem is known as a nonparametric regression problem. In this work, we shall restrict  ourselves to this case. Following the standard notations (see for example \cite{Smale1, Smale2}) and letting $\rho_X$ denote the marginal probability measure over $\mathcal X,$ the true regression function associated with this regression problem is given by 

\begin{equation*}
f(x)= \int_{\mathcal Y} y d \rho(y|x),\quad x\in \mathcal X,
\end{equation*}
where $d \rho(y|x)$ is the conditional distribution of $Y$ given $x.$ It is well known (see for example \cite{Smale1, Tarres, Vito}) that $f_\rho$ minimizes the mean square error $\int_{\mathcal X\times \mathcal Y} (y-f(x))^2 d\rho$. In practice, the outputs $Y_i$ are noised observations of the true regression function, which we will simply denote, from now on, by $f$. Therefore, we consider the following nonparametric regression model:
\begin{equation*}
Y_i = f(X_i) +\eta_i,\quad 1\leq i\leq n,
\end{equation*}
where $(X_i)_{1\leq i \leq n}$ are random design variables (or inputs) with distribution $\rho_X$ and the noise terms $(\eta_i)_{1\leq i \leq n}$ are i.i.d. real-valued random variables with mean zero. For simplicity, we will assume that the $X_i$ are uniformly distributed on the interval $I=[-1, 1]$. The problem is to estimate the function $f : I \rightarrow \mathbb R$, based on observations $(X_i, Y_i)_{1\leq i \leq n}$. In this problem, an important issue is to achieve a trade-off between minimization of the empirical  regression  error $\frac{1}{n} \sum_{i=1}^n \big(f(x_i)-y_i\big)^2 $ and data overfitting, which generally yields  large regression errors. Roughly speaking, an algorithm that  overfits data requires too many inputs and makes predictions that are largely based on  the noise, rather than on the data corresponding  to the regression function itself. Tikhonov regularized least-square algorithm is a  popular learning algorithm that overcomes the overfitting problem and provides a satisfactory approximation to the regression function. More precisely, for an appropriate choice of a Hilbert space $\mathcal H$ and a given regularization parameter $\lambda >0$, Tikhonov regularized least-square estimator of $f$ is given by:
\begin{equation*}
f^\lambda=\arg\min_{f\in \mathcal H} \left\{\frac{1}{n} \sum_{i=1}^n \Big(f(X_i)-Y_i \Big)^2 +\lambda \|f\|_{\mathcal H}^2\right \},
\end{equation*}
where $\|\cdot\|_{\mathcal H}$ is the usual norm of $\mathcal H.$ In practice, the regularization parameter $\lambda$ is chosen in such a way that the resulting  mean square error is small. Moreover, to prevent data overfitting, $\lambda$ must not be too small. A variety of procedures have been proposed for choosing the appropriate value of $\lambda$, such as (generalized) cross-validation, use of a validation set or derivation of an explicit optimal value of $\lambda$ according to some optimality criterion associated with the regression error (we refer the interested reader to \cite{Smale1, Tarres, Vito, Yuan}, for example, for more details on these procedures).

Our aim in this work is to develop some approximation schemes that provide convenient and stable estimates of $f$, provided that this later satisfies some smoothness property. The estimates of $f$ investigated here are supposed to belong to a reproducing kernel Hilbert space (RKHS) generated by a Mercer kernel $K(t,x)$ defined on $\mathcal X\times \mathcal X,$ for some compact set $\mathcal X$ of $\mathbb R$. Note that such a kernel is a real valued, continuous, symmetric and positive semi-definite function. The error analysis of an RKHS regression approximation scheme depends heavily of the special spectral properties of the integral operator associated with the reproducing kernel. In this work, we restrict ourselves to the Sinc and Legendre Christoffel-Darboux (Legendre for short) kernels. For a real positive number $c$ (called bandwidth), the Sinc kernel is defined by
\begin{equation}\label{sinc.kernel}
K_c(x,y)=\frac{\sin(c(x-y))}{\pi (x-y)},\quad x,y\in I.
\end{equation}
For a positive integer $N,$ Legendre kernel is given by 
\begin{equation*}
K_N(x,y)=\frac{N+1}{\sqrt{2N+1}\sqrt{2N+3}}\frac{\widetilde P_{N+1}(x)\widetilde P_N(y)-\widetilde P_N(x)\widetilde P_{N+1}(y)}{x-y},\quad x,y\in I,
\end{equation*}
where the $\widetilde P_k$ are the usual orthonormal Legendre polynomials of degree $k$, $k=0,1,\ldots$ More details on these two Mercer kernels are given in section \ref{maths.prelim}, along with some useful spectral properties of the associated integral operators.

Before giving our plan, let us highlight our main findings. First, we show that the empirical projection operators associated to the Legendre and Sinc kernels provide stable and fairly accurate approximations to the true regression function $f.$  To establish this result, we need to assume some regularity on $f$. Precisely, we assume that $f$ belongs to a fractional Sobolev space $H^s(I)$ for some $s>0$, or that $f$ is the restriction to $I$ of a $c-$bandlimited function $\widetilde f,$ for some $c>0$ (that is, $\widetilde f$ belongs to the Paley-Wiener space $\mathcal B_c,$ defined as the set of functions of $L^2(\mathbb R)$ with Fourier transforms supported on  the interval $[-c,c]$). More precisely, we define our estimator of the regression function $f$, based on Legendre kernel, by:
\begin{equation*}
\widehat f_{N,n}(x)= \frac{2}{n}\sum_{i=1}^n Y_i \sum_{k=0}^N \left[\widetilde{P}_k(X_i)\widetilde{P}_k(x)\right]=\frac{2}{n}\sum_{i=1}^n Y_i K_N(X_i,x),\quad x\in I.
\end{equation*}
By using Pinelis concentration inequality and some spectral approximation properties of Legendre kernel, we prove that if $f\in H^s(I)$ for some $s>0,$ then, for any $0<\delta <1$, there exists a uniform constant $c_1$ such that the following holds with probability at least $1-\delta$:
\begin{equation*}
\|f -\widehat f_{N,n}\|_{L^2(I)} \leq \frac{M_{f,N}}{\sqrt{n}}\sqrt{\log\Big(\frac{2}{\delta}\Big)}+ c_1 N^{-s} \|f\|_{H^s},
\end{equation*}
where $M_{f,N}= 2(N+1) (\|f\|_\infty+\varepsilon)+\sqrt{2} \|f\|_{L^2(I)}$ and $\varepsilon= \max_i |\eta_i|$. Moreover, under the assumption that for some $c>0,$ $f$ is the restriction to $I$ of a $c-$bandlimited function $\widetilde f,$ we prove that there exists a uniform constant $c_2$ such that  for any $0<\delta <1$, the following holds with probability at least $1-\delta$:
\begin{equation*}
\|f -\widehat f_{N,n}\|_{L^2(I)} \leq \frac{M_{f,N}}{\sqrt{n}}\sqrt{\log\Big(\frac{2}{\delta}\Big)}+ c_2 e^{-(N+2) \log\Big(\frac{2N+2}{ec}\Big)}\|\widetilde f\|_{L^2(\mathbb R)}.
\end{equation*}
For a positive real number $c>0,$ we also consider a second  estimator of the regression function $f$, based on the Sinc kernel $K_c$ with bandwidth $c$. This estimator is given by: 
\begin{equation*}
\widehat f_{c,n}(x)= \frac{2}{n}\sum_{i=1}^n Y_i  K_c(X_i,x),
\end{equation*}
where $K_c(\cdot, \cdot)$ is given by \eqref{sinc.kernel}. Under the assumption that $f$ belongs to some weighted Sobolev space $\widetilde H^s,\, s>0,$ we prove that for any $0<\delta <1$, the following holds with probability at least $1-\delta$:
\begin{equation*}
\|f- \widehat f_{c,n}\|_{L^2(I)} \leq   \frac{M_{f,c}}{\sqrt{n}}\sqrt{\log\Big(\frac{2}{\delta}\Big)} + \frac{7}{\sqrt{6}} \left(\frac{e^2}{6}\right)^{-[c/3]}\|f\|_{L^2(I)}+\left[\frac{c}{3}\right]^{-s} \|f\|_{\widetilde H^s},
\end{equation*}
where $M_{f,c}=\sqrt{\frac{2c}{\pi}} \big(2\|f\|_\infty+2\varepsilon+\sqrt{2} \|f\|_{L^2(I)}\big)$.

In the second part of this work, we briefly describe Tikhonov regularized least-square minimization algorithm for the stable approximation of the regression function $f$. The algorithm requires the inversion of a random Gram-matrix, and we show that besides preventing data overfitting, the regularization parameter $\lambda>0$ is crucial to obtain a stable solution. In particular, we check that  the $2-$condition number of  the associated regularized  random Gram-matrix is of order $O(1 \slash \lambda).$  A special interest is given to the Sinc kernel. For this case, we give a fairly precise $2-$condition number of the associated  random Gram-matrix. Moreover, we check that if the regression function $f$ is the restriction to $I$ of a $c-$bandlimited function, then ${\displaystyle  \| \widehat f^\lambda_{c,n}- f\| = O (\sqrt{c} \slash n^{1/4})}$ with high probability, where $\widehat f^\lambda_{c,n}$ is the approximation of $f$ obtained by solving the regularized least square minimization problem with the optimal theoretical admissible value of $\lambda$.

The paper is organized as follows. In section \ref{maths.prelim}, we give some mathematical preliminaries that will be useful in this work. In section \ref{projection.operators}, we define our projection estimators based on Legendre and Sinc kernels and we provide their error analysis. These error analyses are performed under the assumption that the regression function $f$ belongs to some fractional order Sobolev space over $I$ or $f$ is the restriction to $I$ of a bandlimited function. In section \ref{Tikhonov}, we briefly describe Tikhonov regularized least-square minimization algorithm for the approximation of the regression function. We check the stability and convergence rate of the algorithm in the case where the reproducing kernel is given by the Sinc kernel, and the regression function is the restriction to $I$ of a bandlimited function. In section 5, we provide various numerical simulations to illustrate our results.

\section{Mathematical preliminaries}\label{maths.prelim}

In this section, we first provide some preliminaries about Legendre polynomials and Legendre spectral approximations. Then we briefly recall the definition and main properties of a reproducing kernel and its associated RKHS. Specific attention is given to the Sinc kernel and to some spectral properties of the associated integral operator. \\ In the following, unless otherwise stated, the norm $\|\cdot\|_{L^2(I)}$ and usual inner product $<\cdot,\cdot>_{L^2(I)}$ in $L^2(I)$ are denoted by $\|\cdot\|$ and $<\cdot,\cdot>$ respectively. Also,  we denote by $\mathbb E(\cdot)=\mathbb E_{X,\eta}(\cdot)$ and $\mathbb E_X(\cdot)$ the expectations with respect to the r.v. $(X_i,\eta_i)$ and $X_i$ respectively. Finally, for any real number $x$, $[x]$ denotes the integer part of $x$.

\subsection{Legendre polynomials and basis}

For any integer $n\geq 0,$ the normalized  Legendre polynomial $\widetilde P_n$ of degree $n$ is given by Rodrigues formula ${ \widetilde P_{n}(x)=\sqrt{n+\frac{1}{2}}\frac{1}{2^n\,n!} \frac{d^n}{dx^n}\left(\big(x^2-1\big)^n\right)}$, and is such that:
\begin{equation}\label{Legendre2}
\sup_{x\in I} |\widetilde P_{n}(x)| = |\widetilde P_{n}(1)|=\sqrt{{n+\frac{1}{2}}}.
\end{equation}
Let $N$ be a positive integer. The Christoffel-Darboux kernel associated with Legendre polynomials is given by:
\begin{equation}
\label{Christoffel}
K_N(x,y)=\sum_{j=0}^N \widetilde P_j(x) \widetilde P_j(y)= 
\left\{ \begin{array}{ll} \frac{N+1}{\sqrt{2N+1}\sqrt{2N+3}}\frac{\widetilde P_{N+1}(x) \widetilde P_N(y)-\widetilde P_N(x) \widetilde P_{N+1}(y)}{x-y},&\quad x\neq y\\
\frac{N+1}{\sqrt{2N+1}\sqrt{2N+3}}\Big(\widetilde P'_{N+1}(x)\widetilde P_N(x)-\widetilde P'_N(x)\widetilde P_{N+1}(x)\Big),&\quad x=y.\end{array}\right.
\end{equation}
We refer the reader to \cite{NIST} for more details on Legendre polynomials.

The collection $(\widetilde{P_n})_{n\geq 0}$ constitutes an orthonormal basis of the Hilbert space $L^2(I)$ endowed with its usual inner product. This basis is well suited to the spectral approximation of functions belonging to Sobolev spaces $H^s(I)$, where $s>0$ is a positive real number (also known as fractional Sobolev spaces). Such spaces can be defined in several ways. A first approach defines $H^{s}(I)$ as an intermediate space between the classical Sobolev spaces $H^{[s]}(I)$ and $H^{[s]+1}(I)$, via an interpolation technique.  Another approach is to define $H^{s}(I)$ as the set of functions $u$ such that
\begin{equation*}
\| u\|_{H^{s}(I)}=\Vert u\Vert _{H^{[s]}(I)}+ \left[ \int_{I\times I}\frac{\mid v(x)-v(y)\mid ^{2}}{\mid x-y\mid ^{(1+2\sigma )}} dxdy\right] ^{\frac{1}{2}}<\infty
\end{equation*}
where $v=\frac{d^{[s]}u}{dx^{[s]}}$ and $\sigma :=s-[s]$.

Let $\pi_N$ denote the orthogonal projection on the finite-dimensional subspace of $L^2(I)$ spanned by $(\widetilde P_0, \widetilde P_1,\ldots,\widetilde P_N)$. That is, for $f \in L^2(I)$:
\begin{equation*}
\pi_N(f) = \sum_{k=0}^{N} <f,\widetilde P_k>_{L^2(I)}\widetilde P_k, \quad \mbox{ with } \quad <f,\widetilde P_k>_{L^2(I)}=\int_I f(x) \widetilde P_k(x)\, dx.
\end{equation*}
Then it is known (see \cite{Canuto}, for example) that there exists a uniform constant $C>0$ such that:
\begin{equation*}
\|u-\pi_N(u)\|_{L^2(I)} \leq C N^{-s} \|u\|_{H^s(I)},\qquad\forall\,\, u\in H^s(I).
\end{equation*}
Moreover, it was recently shown that Legendre polynomials  are also well suited to the approximation of $c$-bandlimited functions, where $c>0$ is a positive real number. The space $\mathcal B_c$ of $c$-bandlimited functions is defined as: 
\begin{equation}\label{Bc}
\mathcal B_c=\{ f\in L^2(\mathbb R),\,\, \mbox{supp}(\mathcal F f)\in [-c,c]\},
\end{equation}
where $\mathcal F f$ denotes the usual Fourier transform of $f\in L^2(\mathbb R)$. In \cite{JKS}, it is shown that there exists a uniform constant $C>0$ such that  for any integer $N \geq \max(3,ec/2)$ and any $ f \in \mathcal{B}_c$, we have:
\begin{equation*}
\norm{f-\pi_Nf}_{L^2(I)} \leq C  \left(\dfrac{ec}{2N+2}\right)^{N+2} \norm{f}_{L^{2}(\mathbb{R})}.
\end{equation*}

\subsection{Reproducing kernels and associated RKHS}

In this section, we briefly recall the definition and main properties of a reproducing kernel and its associated RKHS. Special attention is given to the Sinc kernel and some of the spectral properties of its associated integral operator.

Let $\mathcal X$ be a measurable set of $\mathbb R$. Then a complex-valued function $K(\cdot,\cdot)$ defined on $\mathcal X\times \mathcal X$ is said to be a reproducing kernel of a Hilbert space $\mathcal H_K$ endowed with an inner product $<\cdot,\cdot>_K$, if 
\begin{equation*}
K_x(\cdot)= K(\cdot, x) \in \mathcal H_K,\quad \forall\, x\in\mathcal X\qquad\mbox{and}\qquad  f(x)=<f, K_x>_K,\quad \forall \, f\in \mathcal H_K. 
\end{equation*}
Such a Hilbert space $\mathcal H_K$ is called a reproducing kernel Hilbert space (RKHS). Moreover, from Riesz representation theorem, for an arbitrary set $\mathcal X$ of $\mathbb R$, the  Hilbert space of real-valued functions on $\mathcal X$ is a RKHS whenever the evaluation linear functional $L_x: f\rightarrow f(x)$ is continuous on $\mathcal H$ for every $x\in \mathcal X$. It is also known that a kernel $K(\cdot,\cdot)$ is a reproducing kernel if and only if it is Hermitian and positive definite, that is:
\begin{equation*}
\sum_{i,j=1}^n c_i \overline{c_j} K(x_i,x_j)\geq 0,\quad \mbox{for any } n\in \mathbb N,\, x_1,\ldots x_n\in \mathcal X,\,\, c_1,\ldots,c_n\in \mathbb C.
\end{equation*}

Next, let $\mathcal X$ be a compact set of $\mathbb R$. Then, a real-valued kernel $K(\cdot,\cdot)$ defined on $\mathcal X\times \mathcal X$ is said to be a Mercer's kernel if it is continuous and positive semi-definite. Moreover, let $\mu$ be a positive measure on $\mathcal X,$  $K(\cdot,\cdot)\in L^2(\mathcal X\times \mathcal X, d\mu\otimes d\mu)$ and $(\varphi_n)_{n\geq 0},  (\lambda_n)_{n\geq 0}$ denote the orthonormal eigenfunctions and associated eigenvalues of the associated Hilbert-Schmidt (thus compact) operator $T_K \varphi(x) = \int_X K(x,y)\varphi_n(y)\, d\mu(y)= \lambda_n \,\varphi_n(x)$, $x\in \mathcal X$. Then, by Mercer's Theorem, we have: $ K(x,y)=\sum_{n=0}^\infty \lambda_n \,\varphi_n(x) \varphi_n(y)$ for $x,y\in \mathcal X$. 

The previous sum converges  uniformly over the compact set $\mathcal X\times \mathcal X.$ Moreover, in this case, the RKHS  $\mathcal H_K$ associated with Mercer's kernel $K(\cdot,\cdot)$ is given by:
\begin{equation*}
\mathcal H_K=\left\{ f\in L^2(\mathcal X,d\mu),\,\, f=\sum_{n\geq 0} a_n(f) \varphi_n,\,\, \sum_{n\geq 0} \frac{|a_n(f)|^2}{\lambda_n}<+\infty\right\}.
\end{equation*}
The associated inner product is given by 
\begin{equation*}
<f,g>_K = \sum_{n\geq 0} \frac{a_n(f) b_n(g)}{\lambda_n},\quad  \mbox{ if } f=\sum_{n\geq 0} a_n(f)\varphi_n,\,\, g=\sum_{n\geq 0} b_n(g)\varphi_n.
\end{equation*}
For  the Sinc convolution kernel, defined for a fixed real number  $c>0$ by $K_c(x,y)= \frac{\sin(c(x-y))}{\pi (x-y)}$ $(x,y \in \mathbb R$), we note that $K_c(x,y)=\frac{c}{\pi} \mathcal F \mu_c(x-y),$ where $\mu_c$ is the uniform probability measure $\mu_c(x)=\frac{1}{2c} \mathbf 1_{[-c,c]}(x)$, $x\in \mathbb R$. Hence, by Bochner's theorem, the Sinc kernel  $K_c(\cdot,\cdot)$ is a reproducing kernel. It is well known (see for example \cite{Slepian1}) that when the Sinc kernel is defined on $\mathbb R^2,$  the associated  RKHS is the Paley-Wiener space of $c-$bandlimited functions $ \mathcal B_c$, given by \eqref{Bc}. We should mention that when the Sinc kernel is restricted to the square $I^2=[-1,1]^2,$ one gets a reproducing kernel which is also a Mercer's kernel, given by: 
\begin{equation*}
K_c(x,y)= \frac{\sin(c(x-y))}{\pi (x-y)},\quad x,y \in I=[-1,1].
\end{equation*}
In this case, we have 
\begin{equation}\label{Sinc_kernel2}
K_c(x,y)= \frac{\sin(c(x-y))}{\pi (x-y)}=\sum_{n=0}^\infty \lambda_n(c) \psi_{n,c}(x) \psi_{n,c}(y),\quad\forall\,  x,y \in I=[-1,1].
\end{equation}
Here, the $\psi_{n,c}$ and  $\lambda_n(c)$ are the eigenfunctions  and associated positive eigenvalues of the Hilbert-Schmidt operator ${\displaystyle \mathcal Q_c}$ defined on $L^2(-1,1)$ by ${ \mathcal Q_c f(x)=\int_{-1}^1 \frac{\sin(c(x-y))}{\pi(x-y)} f(y)\, dy}$. That is, for any integer $n\geq 0$, 
\begin{equation*}
\int_{-1}^1 \frac{\sin(c(x-y))}{\pi(x-y)} \psi_{n,c}(y)\, dy =\lambda_n(c) \psi_{n,c}(x),\quad x\in I.
\end{equation*}
In the literature, eigenfunctions $\psi_{n,c}$ are known as the prolate spheroidal wave functions (PSWFs). The eigenvalues $\lambda_n(c),\, n\geq 0$ are arranged in the decreasing order. They are known to be simple and to satisfy 
\begin{equation*}
1> \lambda_0(c)> \lambda_1(c)>\cdots >\lambda_n(c)>\cdots
\end{equation*}
The theory and computation of these PSWFs and their associated eigenvalues $\lambda_n(c)$ are due to the pioneering works of  D. Slepian and his collaborators H. Landau and H. Pollack, see for example  \cite{Slepian1}. It is important to note that the sequence of positive eigenvalues $\lambda_n(c)$ has a super-exponential decay rate to zero. Moreover, it was recently shown (see \cite{Bonami-Karoui2}) that for any $1\leq a < 4\slash e$, there exists $N_{c,a}\in \mathbb N$ such that 
\begin{equation*}
 \lambda_n(c) \leq e^{-2n \log\left(\frac{an}{c}\right)},\quad\forall\, \, n\geq N_{c,a}.
\end{equation*}
The constant $a=4\slash e$ is optimal. Finally, in \cite{BJK}, authors provide the following useful non-asymptotic behaviour and decay rate. For any $c>0,$ we have:
\begin{equation}
\label{decay2}
\lambda_n(c)\geq 1-\frac{7}{\sqrt{c}} \frac{(2c)^n}{\fact{n}}e^{-c},\quad \mbox{ for } 0\leq n <\frac{c}{2.7},
\end{equation}
and
\begin{equation}
\label{decay3}
\lambda_n(c)\leq \exp\left(-(2n+1)\log\Big(\frac{2}{ec}(n+1)\Big)\right),\quad\forall\, n\geq \max\left(\frac{ec}{2},2\right).
\end{equation}

\section{Nonparametric  regression by empirical projection operators}\label{projection.operators}

Let $(X_i, Y_i)_{1\leq i \leq n}$ be independent observations of the nonparametric regression model
\begin{eqnarray}\label{Nregression1}
Y_i = f(X_i) + \eta_i, \quad i=1,\ldots,n.
\end{eqnarray}
The random design variables $(X_i)_{1\leq i \leq n}$ are assumed to be independent and uniformly distributed on the interval $I=[-1, 1]$, the noise terms $(\eta_i)_{1\leq i \leq n}$ are i.i.d. real-valued random variables with mean zero, and the two sequences are independent. The problem is to estimate the function $f : I \rightarrow \mathbb R$ from the observations $(X_i, Y_i)_{1\leq i \leq n}$. We assume that $f(\cdot)$ lies in a subspace of the Hilbert space $L^2(I)$.

In this section, we consider two  reproducing kernels, namely the Legendre  Christoffel-Darboux and Sinc kernels. We provide error analysis for the proposed nonparametric regression schemes when  $f$ belongs to the Sobolev space $H^s(I)$, $s>0$, or $f$ is the restriction to $I$ of a bandlimited function.

In what follows, we let $N$ be a positive  integer and $K_N(x,y)=\sum_{k=0}^N\widetilde{P}_k(x) \widetilde{P}_k(y)$ be the Legendre  Christoffel-Darboux kernel given by \eqref{Christoffel}. Let $\pi_N$ be the projection operator on the finite-dimensional subspace of $L^2(I)$ spanned by $(\widetilde P_0, \widetilde P_1,\ldots,\widetilde P_N)$, that is:
\begin{equation*}
\pi_N(f)(x)= <f, K_N(x,\cdot)>=\sum \limits_{k=0}^N <f,\widetilde{P}_k>\widetilde{P}_k(x),\quad x\in I.
\end{equation*}
Based on this empirical projection operator, we define the regression estimator of $f$ as:
\begin{equation*}
\widehat f_{N,n}(x)= \frac{2}{n} \sum_{i=1}^n Y_i K_N(X_i,x),\quad x\in I,
\end{equation*}
for any positive integer $n$. In Theorem \ref{Thm1} below, we show that Legendre kernel is well adapted for nonparametric regression of functions belonging to the Sobolev space $H^s(I)$, $s>0$, or to the Paley-Wiener space $\mathcal B_c$  of $c-$bandlimited functions, defined by \eqref{Bc}. The proof is based on the Hilbert space-valued Pinelis concentration inequality (see \cite{Rosasco, Pinelis}). We recall this result, which plays a central role in the study of the quality of approximation by the proposed empirical projection operators:
\vspace{.2cm}

\noindent \textbf{Pinelis inequality} 
 Let $\xi_1,..., \xi_n$ be independent random variables with values in a separable Hilbert space $\mathcal H$  with norm $\|\cdot\|_{\mathcal H}$. Assume that  $\mathbb E(\xi_i)=0$ and $\|\xi_i\|_{\mathcal H}\leq C$ for every $ i=1,...,n$. Then the following holds for any $\varepsilon>0$: 
\begin{equation*}
 \mathbb{P}\left(\frac{1}{n}\Big\| \sum_{i=1}^{n} \xi_i\Big\|_{\mathcal H}\leq \varepsilon\right) \geq 1-2 e^{-\frac{n\varepsilon^2}{2C^2}}.
\end{equation*}
We are now in position to state our first result:
\begin{theorem}\label{Thm1} 
Under the above notations and hypotheses,  we have  
\begin{equation}
\label{Estimation1}
\mathbb E \big(\widehat f_{N,n}(x)\big)= \pi_N(f)(x),\quad x\in I.
\end{equation}
Moreover,
\begin{itemize}
\item if $f$ is a bounded function belonging to $H^s(I)$ for some $s>0$, then for any $0<\delta <1,$ the following holds with probability at least $1-\delta$:
\begin{equation}
\label{Estimation2}
\| f-\widehat f_{N,n}\|\leq \frac{M_{f,N}}{\sqrt{n}}\sqrt{\log\Big(\frac{2}{\delta}\Big)}+c_1 N^{-s}\|f\|_{H^s},
\end{equation}
where $M_{f,N}= 2(N+1) (\|f\|_\infty+\varepsilon)+\sqrt{2} \|f\|$ and $\varepsilon= \max_i |\eta_i|$.
\item if $f$ is the restriction to $I$ of a function $\widetilde f \in \mathcal B_c$ (the space of $c-$bandlimited functions for some $c>0$), then for any integer $N \geq ec \slash 2$ and $\delta >0$, the following holds with probability at least $1-\delta$:
\begin{equation}
\label{Estimation22}
\| f-\widehat f_{N,n}\|\leq\frac{M_{f,N}}{\sqrt{n}}\sqrt{\log\Big(\frac{2}{\delta}\Big)}+c_2 e^{-(N+2) \log\Big(\frac{2N+2}{ec}\Big)}\|\widetilde f\|_{L^2(\mathbb R)}.
\end{equation}
for some uniform constant  $c_2$.
\end{itemize}
\end{theorem}
\begin{proof}
First, we note that:
\begin{equation*}
\mathbb E_X\left( f(X_i) \widetilde P_k(X_i)\right)= \frac{1}{2} \int_I f(y) \widetilde P_k(y)\, dy=\frac{1}{2} <f,\widetilde P_k>,\quad 1\leq i\leq n.
\end{equation*}
By independence of $X_i$ and $\eta_i$, and using the fact that $\mathbb E(\eta_i)=0$, we also note that $\mathbb E(\eta_i \widetilde P_k(X_i))=0$, $1\leq i\leq n$. Thus we have:
\begin{eqnarray*}
\mathbb E(\widehat f_{N,n}(x))&=& \sum \limits_{k=0}^N \frac{2}{n}\sum \limits_{i=1}^n \mathbb E \Big((f(X_i)+\eta_i)\widetilde{P}_k(X_i)\Big)\widetilde{P}_k(x) \\ 
&=& \sum_{k=0}^N \left(\frac{2}{n}\sum_{i=1}^n \mathbb E_X(f(X_i)\widetilde{P}_k(X_i)) \widetilde{P}_k(x)\right) \\
&=& \sum_{k=0}^N <f,\tilde{P}_k>\tilde{P}_k(x) \\ 
&=&\pi_N(f)(x).
\end{eqnarray*}
Now, we calculate
\begin{eqnarray*}
\widehat f_{N,n}(x)-\mathbb E\big(\widehat f_{N,n}(x)\big)&=& \frac{2}{n}\sum_{i=1}^n Y_i K_N(X_i,x)  -\pi_N(f)(x) \nonumber\\
&=&\frac{1}{n} \sum \limits_{i=1}^n  \sum_{k=0}^N  \left[ 2(f(X_i)+\eta_i)\widetilde{P}_k(X_i)\widetilde{P}_k(x)-  <f,\widetilde{P}_k>\widetilde{P}_k(x)\right] \\
&=&\frac{1}{n} \sum_{i=1}^n \xi_i(x).
\end{eqnarray*}
We have already checked that $\mathbb E(\xi_i)=0.$ Now,
\begin{eqnarray*}
\|\xi_i\|^2 = \left\|\sum_{k=0}^N \Big(2(f(X_i)+\eta_i)\widetilde P_k(X_i)-<f,\widetilde P_k>\Big) \widetilde P_k\right\|^2,
\end{eqnarray*}
and by Parseval's equality, we get 
\begin{eqnarray}\label{pars}
\lefteqn{\left\|\sum_{k=0}^N \Big(2(f(X_i)+\eta_i)\widetilde P_k(X_i)-<f,\widetilde P_k>\Big) \widetilde P_k\right\|^2 =\sum_{k=0}^N \Big(2(f(X_i)+\eta_i)\widetilde P_k(X_i)-<f,\widetilde P_k>\Big)^2} \nonumber \\ 
&&\qquad\qquad \qquad\qquad\qquad\qquad\qquad\qquad\qquad\leq  2 \sum \limits_{k=0}^N \left(4(|f(X_i)|+\eta_i)^2 |\widetilde{P}_k(X_i)|^2+ | <f,\widetilde{P}_k>|^2\right).
\end{eqnarray}
Now, the set $(\widetilde P_k)_{0\leq k\leq N}$ is an orthonormal set of $L^2(I),$ thus by Bessel inequality, we have:
\begin{equation}\label{bessel}
\sum_{k=0}^N |<f,\widetilde P_k>|^2 \leq \|f\|^2.
\end{equation}
Hence, it follows from \eqref{pars}, \eqref{bessel} and \eqref{Legendre2} that: 
\begin{eqnarray*}
\|\xi_i\|^2 &\leq&  8(\|f\|_{\infty}+\varepsilon)^2 \sum_{k=0}^N \Big(k+\frac{1}{2}\Big)+2\|f\|^2 \nonumber \\
&\leq & 4(\|f\|_{\infty}+\varepsilon)^2(N+1)^2+ 2 \|f\|^2,
\end{eqnarray*}
and thus
\begin{equation*}
 \|\xi_i\| \leq C_{f,N} := \Big(4(\|f\|_{\infty}+\varepsilon)^2(N+1)^2+ 2 \|f\|^2\Big)^{1/2},\quad\forall\, 1\leq i \leq n.
\end{equation*}
In view of Pinelis inequality, we have, for any $\varepsilon >0$:
\begin{equation}\label{Estimation3}
\mathbb{P}\left(\|\widehat f_{N,n}-\mathbb E(\widehat f_{N,n})\|\leq \epsilon\right) \geq 1-2 e^{-\frac{n\varepsilon^2}{2 C_{f,N}^2}}.
\end{equation}
Moreover by using \eqref{Estimation1}, we get 
\begin{equation}\label{Eqq2.3}
\|f-\widehat f_{N,n}\|\leq \|f-\pi_N(f)\|+\|\widehat f_{N,n}-\mathbb E(\widehat f_{N,n})\|. 
\end{equation}
On the other hand, it is well-known that if $f \in H^{s}(I)$, we have 
\begin{equation}\label{Eq2.3}
\|f-\pi_N(f)\| \leq c_1 N^{-s} \|f\|_{H^s},
\end{equation}
for some uniform constant $c_1.$
Hence, by combining \eqref{Estimation3},  \eqref{Eqq2.3}  and \eqref{Eq2.3}, we obtain
\begin{equation*}
\mathbb{P}\left(\| f-\widehat f_{N,n}\|\leq\varepsilon+c_1 N^{-s}\|f\|_{H^s}\right)\geq 1- 2e^{-\frac{n\varepsilon^2}{2 C_{f,N}^2}}.
\end{equation*}
By using the substitution $\delta = 2e^{-\frac{n\varepsilon^2}{2 C_{f,N}^2}},$ this inequality can be rewritten as \eqref{Estimation2}. Finally, to prove \eqref{Estimation22}, it suffices to note that if $f\in \mathcal B_c$, we have, for $N\geq \frac{ec}{2}$:
\begin{equation*}
\| f-\pi_N(f)\|\leq c_2 e^{-(N+2) \log\Big(\frac{2N+2}{ec}\Big)}\|f\|_{L^2(\mathbb R)},
\end{equation*}
for some uniform constant $c_2.$ By using similar techniques as above, we obtain \eqref{Estimation22}.
\end{proof}

\begin{remark}
From the error bound given by \eqref{Estimation2}, we conclude that if the regression function $f$ belongs to $H^s(I), \, s>0,$ then the minimum  error of the  estimator $\widehat f_{N,n}$  is obtained when the two error terms in \eqref{Estimation2} are of the same order. Straightforward computation shows that this is the case if $N^{-s}$ has the same order as $\frac{M_{f,N}}{\sqrt{n}}$, that is when $N=O(n^{\frac{1}{2(s+1)}}).$ Consequently, the  convergence rate of $\widehat f_{N,n}$ is $O(n^{\frac{-s}{2(1+s)}}),$ where $s$ is the Sobolev smoothness of $f$.
\end{remark}

Now, we extend the result of Theorem~\ref{Thm1} to the case of the Sinc kernel. For this purpose, we consider a real number $c>0$ and the 
Sinc kernel $K_c(x,y)=\frac{\sin c(x-y)}{\pi(x-y)},\, x,y\in I.$ Let $(\psi_{n,c}(x))_{n\geq 0}, (\lambda_n(c))_{n\geq 0}$ be the set of orthonormal eigenfunctions and  associated eigenvalues of the Sinc kernel operator $\mathcal Q_c,$ given by $\mathcal Q_c (f)(x)=<K_c(x,\cdot),f>$. Let $\pi_c$ denote the projection operator over the subspace spanned by the $\psi_{n,c}$. Then, by using the fact that the $(\psi_{n,c}(\cdot))_{n\geq 0}$ form an orthonormal basis of $L^2(I),$ we obtain
\begin{equation}\label{Eq2.4}
f(x)=\pi_c(f)(x)= \sum \limits_{n=0}^{\infty} <f,\psi_{n,c}>\psi_{n,c}(x),\quad \forall\, f\in L^2(I).
\end{equation}
The regression estimator of $f$ in model \eqref{Nregression1}, given by the  empirical projection operator associated with the Sinc kernel, is defined by:
\begin{equation}\label{Eq2.5}
 \widehat f_{c,n}(x)=\widetilde{\pi}_{c,n}(f)(x)= \frac{2}{n}\sum_{i=1}^n Y_i K_c(x,X_i)=\frac{2}{n}\sum_{i=1}^n (f(X_i)+\eta_i) K_c(x,X_i).
\end{equation}
The Sinc kernel defined on $I^2$ is a Mercer's kernel. Then assuming that $f$ is bounded on $I$ and using \eqref{Sinc_kernel2} and the fact that $\mathbb E(\eta_i K(X_i,x))=0$, we calculate: 
\begin{eqnarray}\label{Eq2.6}
\mathbb E (\widehat f_{c,n}(x))&=& \mathbb E(\widetilde{\pi}_{c,n}(f)(x))=\int_{I} K_c(x,y)f(y)\, dy = \int_{I}\sum_{k=0}^{\infty} \lambda_k(c) \psi_{k,c}(x) \psi_{k,c}(y) f(y)\, dy \nonumber \\
&=& \sum_{k=0}^{\infty} \lambda_k(c)\left(\int_{I}  \psi_{k,c}(y) f(y)\, dy\right)\psi_{k,c}(x)=\sum_{k=0}^{\infty} \lambda_k(c)<f,  \psi_{k,c}>\psi_{k,c}(x).
\end{eqnarray}
Note that  $\mathbb E(\widehat f_{c,n})\neq \pi_c(f).$ Consequently, the  approach we used for analysing the Legendre Christoffel-Darboux based regression scheme cannot be applied to the Sinc kernel. To overcome this difficulty, we first substitute the usual Sobolev space $H^s(I)$ with the weighted Sobolev space $\widetilde H^s(I),$ defined  by
\begin{equation*}
\widetilde H^s(I)=\left\{ f\in L^2(I),\,\, \|f\|^2_{\widetilde H^s}=\sum_{k\geq 0} (1+k^2)^s |<f,\psi_{k,c}>|^2 <+\infty\right\}.
\end{equation*}
It is easy to check that if $f\in \widetilde H^s(I),$ then for any positive integer $N,$
\begin{equation}
\label{Eq2.7}
\sum_{n=N+1}^\infty  |<f,\psi_{n,c}>|^2 =\| f-\pi_{c,N} (f)\|^2\leq N^{-2s}  \|f\|^2_{\widetilde H^s}.
\end{equation}
We first establish two technical lemmas that will be needed to provide an error analysis of the Sinc kernel based regression scheme for a regression function belonging to the weighted Sobolev space $\widetilde H^s(I)$.

\begin{lemma}\label{Lem1}
For any real number $c\geq 6$ and any positive integer $N$ such that $N+1\leq \frac{c}{3},$ we have 
\begin{equation}
\label{Eq2.9}
\frac{e^{-c}}{\sqrt{c}} \sum_{k=0}^N \frac{(2c)^k}{k!}\leq \frac{1}{\sqrt{6}} \left(\frac{e^2}{6}\right)^{-c/3}.
\end{equation}
\end{lemma}

\begin{proof} First, note from \cite{Batir} that $k!= \Gamma(k+1) \geq \sqrt{2e} \left(\frac{2k+1}{2e}\right)^{k+\frac{1}{2}}$. Therefore:
\begin{equation}\label{majoration}
\frac{e^{-c}}{\sqrt{c}} \sum_{k=0}^N \frac{(2c)^k}{k!} \leq e^{-c} \sum_{k=0}^N \frac{1}{\sqrt{(2k+1)c}} \left(\frac{4ec}{2k+1}\right)^k.
\end{equation}
If $0\leq k\leq N$ with $N+1\leq \frac{c}{3},$ the quantity in the sum above is an increasing function of $k.$ Thus, the right-hand side of \eqref{majoration} is bounded as follows:
\begin{eqnarray*}
e^{-c}\sum_{k=0}^N \frac{1}{\sqrt{(2k+1)c}} \left(\frac{4ec}{2k+1}\right)^k &\leq& e^{-c}\frac{N+1}{\sqrt{(2N+1)c}} \left(\frac{4ec}{2c/3}\right)^{c/3}\\
&\leq & e^{-c} \frac{1}{\sqrt{6}} (6e)^{c/3}=\frac{1}{\sqrt{6}}\left(\frac{e^2}{6}\right)^{-c/3}.
\end{eqnarray*} 
\end{proof}

\begin{lemma}\label{Lem2}
Let $f\in \widetilde H^s(I),\, s>0$. Then for any real $c\geq 6,$  we have (with notations as above):
\begin{equation}\label{Eq2.10}
\| f- \mathbb E (\widehat f_{c,n})\|\leq \frac{7}{\sqrt{6}} \left(\frac{e^2}{6}\right)^{-[c/3]}\|f\|+\left[\frac{c}{3}\right]^{-s} \|f\|_{\widetilde H^s}.
\end{equation}
\end{lemma}

\begin{proof}
Let $N$ be a positive integer such that $N\leq \left[\frac{c}{3}\right]$. It follows from \eqref{Eq2.4} and \eqref{Eq2.6} that:
\begin{eqnarray}\label{Eqq2.10}
\lefteqn{\left\|f-\mathbb E(\widehat f_{c,n})\right\|=\Big\|f- \sum \limits_{n=0}^{\infty} \lambda_n(c)<f,\psi_{n,c}>\psi_{n,c}\Big\|
=\Big\|\sum \limits_{n=0}^{\infty} \big(1-\lambda_n(c)\big)<f,\psi_{n,c}>\psi_{n,c}\Big\|}\nonumber\\
&\leq & \Big\|\sum \limits_{n=0}^{N} \big(1-\lambda_n(c)\big) <f,\psi_{n,c}>\psi_{n,c}\Big\| + \Big\|\sum \limits_{n=N+1}^{\infty} (1-\lambda_n(c))<f,\psi_{n,c}> \psi_{n,c}\Big\|.
\end{eqnarray}
To bound the first term in the right-hand side of \eqref{Eqq2.10}, we proceed as follows. Cauchy-Schwarz inequality and the fact that $\|\psi_{n,c}\|=1$ imply: $|<f,\psi_{n,c}>|\leq \|f\| \|\psi_{n,c}\| \leq  \|f\|.$ Moreover, $1-\lambda_n(c)>0$ for $n\geq 0$. Thus by using Minkowski inequality, we obtain: 
\begin{equation*}
\Big\|\sum_{n=0}^{N} \big(1-\lambda_n(c)\big) <f,\psi_{n,c}> \psi_{n,c}\Big \| \leq \sum_{n=0}^N  \big(1-\lambda_n(c)\big) \|f\|.
\end{equation*}
From \eqref{decay2}, for any $0\leq n\leq N\leq c/2.7,$ we have ${\displaystyle 1-\lambda_n(c)\leq \frac{7}{\sqrt{c}} e^{-c} \frac{(2c)^n}{n!}}$. Combining this with the previous inequality and inequality \eqref{Eq2.9} of Lemma \ref{Lem1}, we obtain, for $N=[\frac{c}{3}]$:
\begin{equation}\label{Eq2.11}
\Big\|\sum \limits_{n=0}^{N} \big(1-\lambda_n(c)\big) <f,\psi_{n,c}>\psi_{n,c}\Big\|\leq \frac{7}{\sqrt{6}} \left(\frac{e^2}{6}\right)^{-[c/3]}\|f\|.
\end{equation}
To bound the second term in the right-hand side of \eqref{Eqq2.10}, we use again the fact that $0< \lambda_n(c)<1$ for $n\geq 0.$ Then, using Parseval's equality yields:
$$ \Big\|\sum \limits_{n=N+1}^{\infty} \big(1-\lambda_n(c)\big)<f,\psi_{n,c}> \psi_{n,c}\Big\|^2=\sum \limits_{n=N+1}^{\infty} \big(1-\lambda_n(c)\big)^2|<f,\psi_{n,c}>|^2\leq \sum \limits_{n=N+1}^{\infty}|<f,\psi_{n,c}>|^2.$$
Hence, by using  \eqref{Eq2.7}  with $N=\left[\frac{c}{3}\right],$ we obtain
\begin{equation}\label{Eq2.12}
\Big\|\sum \limits_{n=N+1}^{\infty} (1-\lambda_n(c))<f,\psi_{n,c}> \psi_{n,c} \Big\|\leq \left[\frac{c}{3}\right]^{-s} \|f\|_{\widetilde H^s}.
\end{equation}
Finally, combining \eqref{Eqq2.10}, \eqref{Eq2.11} and \eqref{Eq2.12} concludes the proof.
\end{proof}

The next theorem quantifies the quality of the Sinc kernel based regression scheme in the weighted Sobolev space.

\begin{theorem}\label{Thm2}
Let $c\geq 6$ and $s>0$ be positive real numbers. Assume that the regression function $f$ in model \eqref{Nregression1} is bounded on $I$ and belongs to $\widetilde H^s(I).$  Then, for any $0<\delta <1,$ the following holds with probability at least $1-\delta$:
\begin{equation}\label{Error2}
\| f-\widehat f_{c,n}\|\leq \frac{M_{f,c}}{\sqrt{n}}\sqrt{\log\left(\frac{2}{\delta}\right)} + \frac{7}{\sqrt{6}} \left(\frac{e^2}{6}\right)^{-[c/3]}\|f\|+\left[\frac{c}{3}\right]^{-s} \|f\|_{\widetilde H^s},
\end{equation}
where $M_{f,c}=\sqrt{\frac{2c}{\pi}}(2\|f\|_\infty+2\varepsilon+\sqrt{2} \|f\|)$ and $\varepsilon=\max_i|\eta_i|$. 
\end{theorem}

\begin{proof} We have
\begin{equation}\label{Eq2.13}
\|f-\widehat f_{c,n}\|\leq \| f-\mathbb E(\widehat f_{c,n})\|+\|\mathbb E(\widehat f_{c,n})-\widehat f_{c,n}\|.
\end{equation}
Now, using \eqref{Eq2.5} and \eqref{Eq2.6}, we can write:
\begin{eqnarray*}
\widehat f_{c,n}(x)-\mathbb E(\widehat f_{c,n}(x))= \frac{1}{n} \sum_{i=1}^n \Big( 2 K_c(X_i,x) (f(X_i)+\eta_i)-<f, K_c(\cdot,x)> \Big)=\frac{1}{n}\sum_{i=1}^n \xi_i(x),
\end{eqnarray*}
where
\begin{equation}
\label{EEq2.13}
|\xi_i(x)| \leq 2 |K_c(X_i,x)| (|f(X_i)|+|\eta_i|)+|<f,K_c(\cdot,x)>|,\,\, x\in I.
\end{equation}
Next, we check that
\begin{equation}
\label{Eqq2.13}
\| K_c(x,\cdot)\| \leq \sqrt{\frac{c}{\pi}},\quad \forall \, x\in I.
\end{equation}
To see this, note that 
\begin{eqnarray*}
\int_I \left(\frac{\sin c(x-y)}{\pi(x-y)}\right)^2 \, dy =\int_{x-1}^{x+1}\left(\frac{\sin c t}{\pi t}\right)^2 \, dt \leq \int_{\mathbb R} \left(\frac{\sin c t}{\pi t}\right)^2 \, dt.
\end{eqnarray*}
Since $\frac{\sin ct}{\pi t}=\frac{1}{2\pi} \mathcal F(\mathbf 1_{[-c,c]}(\cdot))(t)$ (with $\mathcal F$ the usual Fourier transform), Plancherel's equality implies
\begin{equation*}
\int_{\mathbb R} \left(\frac{\sin c t}{\pi t}\right)^2 \, dt= \frac{1}{4\pi^2} 2\pi \int_{\mathbb R} \Big(\mathbf 1_{[-c,c]}(t)\Big)^2 \, dt =\frac{c}{\pi},
\end{equation*}
which achieves the proof of \eqref{Eqq2.13}. Next, using \eqref{EEq2.13} and \eqref{Eqq2.13} and applying Cauchy-Schwarz inequality to the second term in \eqref{EEq2.13}, we can easily check that:
\begin{equation*}
\|\xi_i\| \leq C_{f,c}:=\sqrt{\frac{c}{\pi}} \Big(2 (\|f\|_\infty+\varepsilon)+\sqrt{2}\|f\|\Big),\quad \mbox{ where }\varepsilon =\max_i |\eta_i|.
\end{equation*}
Applying Pinelis concentration inequality, one gets, for any $\varepsilon >0$:
$$ \mathbb{P}\left(\| \widehat f_{c,n}-\mathbb E(\widehat f_{c,n})\|\leq\varepsilon\right)\geq 1- 2e^{-\frac{n\varepsilon^2}{2 C_{f,c}^2}}.$$
Finally, Theorem \ref{Thm2} is proved if we combine this inequality with \eqref{Eq2.10} (Lemma \ref{Lem2}), \eqref{Eq2.13} and if we take ${\displaystyle \delta = 2e^{-\frac{n\varepsilon^2}{2 C_{f,c}^2}}}$. 
\end{proof}

\begin{remark}
Under the assumption that the bandwidth $c>0$ is large enough and using \eqref{Error2}, we conclude that the Sinc kernel based empirical projection operator attains its minimum error when the first and third terms in \eqref{Error2} are of the same order. Straightforward computation shows that this is the case when ${\displaystyle \sqrt{c}=O(n^{\frac{s}{2s+1}}).}$ Hence, the convergence rate of $\widehat f_{c,n}$ is ${\displaystyle O(n^{\frac{-2s+1}{4s}}),}$ where $s>0$ denotes the Sobolev smoothness of $f.$
\end{remark}

\section{Regression estimator by Tikhonov regularized least square minimization over an RKHS}\label{Tikhonov}
 
In this section, we first briefly describe a learning scheme based on Tikhonov regularized minimization over an RKHS, with  associated  Mercer's kernel $K(\cdot,\cdot)$ (this description is based on references \cite{Smale1, Smale2}). This scheme is used as a tool for constructing an appropriate regression function $f\in \mathcal H_K$ from a set of observations $(X_i,Y_i)_{1\leq i\leq n}$ drawn from a joint probability measure $\rho$ on $\mathcal X\times \mathbb R.$  Here, it is assumed that
the $X_i$'s are random observations with values in $\mathcal X$ and probability distribution $\rho_{\mathcal X}$ on $\mathcal X$. Then, we give some details about the stability and convergence of this scheme when Mercer's kernel is given by the Sinc kernel.

In \cite{Smale1, Smale2}, authors consider the true regression function $f_\rho$ defined on $\mathcal X$ by :
\begin{equation*}
f_\rho(x)=\int_\mathbb R y d\rho(y|x),\quad x\in \mathcal X,
\end{equation*}
where $d\rho(y|x)$ is the conditional distribution of $Y$ given $x.$ Then, they consider the Tikhonov regularization algorithm for the approximation of $f_\rho$, based on the data $(X_i,Y_i)_{1\leq i\leq n}$. For a given regularization parameter $\lambda >0,$ this algorithm consists in finding the solution $f_{\lambda}\in \mathcal H_K$ of the minimization problem
\begin{equation}\label{Tikhonov1}
f_{\lambda}=\arg\min_{f\in \mathcal H_K} \left\{\frac{1}{n} \sum_{i=1}^n \Big(f(X_i)-Y_i)\Big)^2 +\lambda \|f\|^2_K\right \},
\end{equation}
where $\|\cdot\|_K$ is the norm associated with the RKHS $\mathcal H_K$ generated by Mercer's kernel $K(\cdot,\cdot).$ In \cite{Smale1, Smale2}, it is shown that a solution of \eqref{Tikhonov1} is given by the estimator
\begin{equation}\label{Estimator3}
\widehat f^\lambda_n(x)= \sum_{i=1}^n c_{i,\lambda} K(x_i,x),
\end{equation}
where the expansion coefficients vector $\mathbf C_\lambda= (c_{i,\lambda})_{1\leq i\leq n}$ is a solution of the system
\begin{equation}\label{Tikhonov3}
\left[ \Big[K(x_i,x_j)\Big]_{1\leq i,j\leq n} + n\lambda I_n\right] \mathbf C_\lambda =G_\lambda\,  \mathbf C_\lambda =\mathbf Y,\quad \mathbf Y= (y_i)_{1\leq i\leq n},
\end{equation}
with $I_n$ the $n\times n$ identity matrix and $[K(x_i,x_j)]_{1\leq i,j\leq n}$ the random Gram-matrix associated with the kernel $K(\cdot,\cdot).$ This result is a consequence of the famous representer theorem. This  theorem is particularly useful in the sense that even if the RKHS associated with the minimization problem  \eqref{Tikhonov1} is of infinite dimension, the solution  \eqref{Estimator3} always lies in a finite dimensional space.
Thanks to this representer theorem, RKHS based schemes are also used for solving  other types of regression problems  such as functional linear regression, see for example \cite{Shin, Yuan}.  
 Also, under the condition that the regularized random Gram matrix $G_\lambda$ given by \eqref{Tikhonov3} is invertible, the expansion coefficients vector $\mathbf C_\lambda$ is given by 
\begin{equation}\label{Tikhonov4}
\mathbf C_\lambda= G_\lambda^{-1} \mathbf Y = \left[ \Big[K(x_i,x_j)\Big]_{1\leq i,j\leq n} + n\lambda I_n\right]^{-1} \mathbf Y,\quad  \mathbf Y= (y_i)_{1\leq i\leq n}.
\end{equation}

\begin{remark}\label{Rem3.1} 
Tikhonov regularization scheme with Mercer's kernel has the advantage to work with random sampling sets $\{X_i,\, i=1,\ldots,n\}$  drawn from a fairly general probability measure $\rho_{\mathcal X}$ on a compact metric space $\mathcal X.$ This is unlike the empirical projection schemes of the previous section, which are restricted to a sampling set following a uniform law on the  compact interval $I.$  This aspect is important in applications where the $X_i$'s follow a fairly general probability distribution on $\mathcal X.$ 
\end{remark}

We recall that the condition number of an $n\times n$ non-singular matrix $A$ with real or complex entries is given by:
$$\kappa(A)= \| A^{-1}\| \| A\|,$$
where $\|\cdot \|$ is a matrix norm. In the special case of the matrix $2-$norm and if the matrix $A$ is  Hermitian,  the $2-$condition number of $A$ is given by $\kappa_2(A)=\frac{|\lambda_{max}(A)|}{|\lambda_{min}(A)|}.$
In the special case where $A$ is given by the random Gram matrix $G_\lambda,$ given by \eqref{Tikhonov3}, it is well known that the stability of the numerical scheme \eqref{Tikhonov4} depends on the magnitude of $\kappa_2(G_\lambda).$ The larger the quantity $\kappa_2(G_\lambda)$ is, the more unstable is the scheme \eqref{Tikhonov4}  and vice-versa. Since $K$ is a positive-definite Mercer's kernel, the matrix $G_0$ (that is for $\lambda=0$) is an $n\times n$ positive-definite random symmetric matrix. Note that for such a kernel $K,$ we have 
$$\sup_{x,y} |K(x,y)|=\sup_x K(x,x) \leq \kappa^2.$$ Consequently, we have 
\begin{equation}\label{Estimatekappa2}
\kappa_2(G_\lambda)\leq 1+ \frac{\kappa^2}{\lambda}.
\end{equation}
This follows from the fact that 
$$ \kappa_2(G_\lambda)=\frac{\lambda_{\max}(G_\lambda)}{\lambda_{\min}(G_\lambda)}\leq \frac{\lambda_{\max}(G_0)+n\lambda}{n\lambda}\leq \frac{\mbox{Trace} (G_0)+n\lambda}{n\lambda}=\frac{ n\kappa^2+ n\lambda}{n \lambda}.$$
Note that this general upper bound is not optimal since it is based on  the rough bound $\lambda_{\max}(G_0) \leq \mbox{Trace} (G_0)$.
\\
In the case of the Sinc kernel, the $n\times n$ Gram  random matrix $G^c_{0,n}$ and its regularized  version $G^c_{\lambda,n}$ are  given by 
\begin{equation}
\label{GramMatrices}
G^c_{0,n} = \Big[ K_c(x_i,x_j)]_{1\leq i,j\leq n} \quad \mbox{ and }\quad G^c_{\lambda,n} = G^c_{0,n} +n \lambda I_n,
\end{equation}
where $K_c(x_i,x_j)=\frac{\sin(c(x_i-x_j))}{\pi (x_i-x_j)}$ and $c$ is a real positive number. Here, the $x_i$'s are assumed to be random observations from the uniform law on $I=[-1,1].$ In the following proposition, we provide estimates for the $2$-condition numbers $\kappa_2(G^c_{0,n})$ and $\kappa_2(G^c_{\lambda,n}).$

\begin{proposition}\label{2-condtion.prop} Let $G^c_{0,n}$ and $G^c_{\lambda,n}$ be the two positive definite 
random matrices given by \eqref{GramMatrices}. Then, the following holds for any $c\geq \frac{5}{2}$:
\begin{equation}\label{Estimate1}
\mathbb E \big(\kappa_2(G^c_{0,n}) \big) \geq \frac{1}{2} e^{2n \log\big(\frac{2n}{e c}\big)},\quad \forall \, n > \frac{e c}{2}.
\end{equation}
Moreover, for any integer $n\geq 1$ and with high probability, we have:
\begin{equation}\label{Estimate2}
\kappa_2(G^c_{\lambda,n}) \leq 1+ \frac{1}{\lambda}\left(1+\frac{c}{\pi \sqrt{n}}\right).
\end{equation}
\end{proposition}

\begin{proof} To prove \eqref{Estimate1}, we first note that for a general positive definite Gram matrix $G_{0,n}=[K(X_i,X_j)]_{1\leq i,j\leq n},$ we have  
\begin{equation*}
\kappa_2(G_{0,n}) = \frac{\lambda_1(G_{0,n})}{\lambda_n(G_{0,n})}=\frac{\lambda_1\big(\frac{1}{n} G_{0,n}\big)}{\lambda_n\big(\frac{1}{n} G_{0,n}\big)}.
\end{equation*}
Since  $\lambda_1\big(\frac{1}{n} G_{0,n}\big)$ and $\lambda_n\big(\frac{1}{n} G_{0,n}\big)$ are independent and $\mathbb E\lambda_n\big(\frac{1}{n} G_{0,n}\big)>0,$
then we have 
$$ \mathbb E(\kappa_2(G_{0,n}))= \mathbb E \Big(\lambda_1\big(\frac{1}{n} G_{0,n}\big)\Big)\mathbb E \left(\frac{1}{\lambda_n\big(\frac{1}{n} G_{0,n}\big)}\right) \geq  \mathbb E \Big(\lambda_1\big(\frac{1}{n} G_{0,n}\big)\Big) \frac{1}{\mathbb E \Big(\lambda_n\big(\frac{1}{n} G_{0,n}\big)\Big)}.$$
This last inequality follows from Jensen inequality and the convexity on $(0,\infty)$ of the function $x\rightarrow \frac{1}{x}.$ 
 On the other hand, it is known (see for example \cite{Shawe})  that  for $1\leq d\leq n$:
 \begin{equation}\label{comparaison}
 \mathbb E\Big(\sum_{j<d} \lambda_j \big(\frac{1}{n} G_{0,n}\big)\Big)\geq \sum_{j<d} \lambda_{j-1}(L_K),\qquad \mathbb E \Big(\sum_{j\geq d} \lambda_j \big(\frac{1}{n} G_{0,n}\big)\Big)\Big)\leq \sum_{j\geq d} \lambda_{j-1}(L_K).
 \end{equation}
 Here, the $\lambda_j(G_{0,n})$ and $\lambda_j(L_K)$ denote the eigenvalues of the random Gram matrix $G_{0,n}$ and its associated integral operator $L_K,$ respectively.  In particular, for the Sinc kernel $K_c,$ $c>0$ and by  letting $d=2$ and $d=n$ in the first and second inequality of \eqref{comparaison} respectively, we obtain 
 $$ \mathbb E(\kappa_2(G^c_{0,n})) \geq \frac{\lambda_0(\mathcal Q_c)}{\lambda_{n-1}(\mathcal Q_c)}.$$
 Now, to conclude for  the proof of \eqref{Estimate1}, it suffices to combine the previous inequality with inequalities \eqref{decay2} and \eqref{decay3} with 
 $c\geq  \frac{5}{2},$ so that the quantity $1-\frac{7}{\sqrt{c}} e^{-c} \geq \frac{1}{2}.$ 

Then, to prove \eqref{Estimate2}, we first recall Weyl's perturbation theorem of the spectrum of a perturbed Hermitian matrix. Let $H$ be an $n\times n $ Hermitian matrix and  $E$ be a Hermitian  perturbation matrix. We assume that the eigenvalues of the Hermitian matrices are arranged in the decreasing order, that is  
$$\mu_1(H+E)\geq \mu_2(H+E)\geq \cdots \geq \mu_n(H+E),\,\, \lambda_1(E)\geq \lambda_2(E)\geq \cdots \geq \lambda_n(E) \mbox{ and } \rho_1(E)\geq \rho_2(E)\geq \cdots \geq \rho_n(E).$$
Then we have  
\begin{equation*}
\lambda_i(H)+\rho_n(E)\leq \mu_i(H+E)\leq \lambda_i(H)+\rho_1(E),\quad i=1,\ldots,n.
\end{equation*}
Next, we use Weyl's perturbation formula with $H= \frac{1}{n} G^c_{0,n}$ and $E= I_n.$ In this case, using the fact that $\frac{1}{n} G_0$ is positive definite, one gets 
 \begin{equation}
 \label{Estimate3}
 \lambda_n \left(\frac{1}{n} G^c_{\lambda,n}\right)\geq \lambda.
 \end{equation}
Moreover, for any $\xi>0$, we have (see \cite{Bonami-Karoui4}):
\begin{equation*}
\mathbb P\left(\left| \lambda_1\left(\frac{1}{n} G^c_{0,n}\right)-\lambda_0(\mathcal Q_c)\right|\leq \frac{c}{\pi} \frac{(1+\xi)}{\sqrt{n}}\right)\geq 1-e^{-\xi^2}.
\end{equation*}
 Since $\lambda_0(\mathcal Q_c)\leq 1,$ this inequality implies that for any $\xi>0$:
\begin{equation*}
\mathbb P\left(\lambda_1\left(\frac{1}{n} G^c_{0,n}\right)\leq 1+ \frac{c}{\pi} \frac{(1+\xi)}{\sqrt{n}}\right)\geq 1-e^{-\xi^2}.
\end{equation*}
Combining this last inequality with \eqref{Estimate3}, one gets  \eqref{Estimate2}.
 \end{proof}

In \cite{Smale2}, authors provide an error analysis of the previous RKHS based scheme for solving the  penalized least square minimization
problem. More precisely, let $\mathcal X $ be  a  compact set and assume that the  Mercer's kernel $K$ is such that  the associated RKHS $\mathcal H_K\subset C(\mathcal X).$ Let also $\kappa^2 =\sup_{x\in \mathcal X} K(x,x)$ and $L_K$ be the integral operator from $L^2(\mathcal X, \rho_\mathcal X)$ to
$\mathcal H_K$ defined by 
$$ L_K(f)(x) =\int_{\mathcal X} K(x,y) f(y) d \rho_{\mathcal X}(y).$$
Since $K$ is a Mercer's kernel, it follows from the spectral decomposition of a
self-adjoint and positive definite compact operator that the fractional power $L_K^{s}$ (for $s\in \mathbb R$) of $L_K$  is defined by
$$L_K^{s}(f)(x) = \sum_{k\geq 0}  \mu_k^s <f, \varphi_k> \varphi_k(x),\quad x\in \mathcal X.$$
Here, $(\mu_k)_{k\geq 0}$ are the eigenvalues of $L_K$ (which are positive) and $\varphi_k$ are the associated orthonormal eigenfunctions. Assume that there exists $M>0$ such that $|f(x)|\leq M$ almost surely in $\mathcal X$ and that 
\begin{equation}
\label{Assumption1}
L_K^{-r} f\rho \in L^2(\mathcal X, d\rho_X), \quad\mbox{ for some } 0<r \leq 1.
\end{equation}
Recall that $f_\rho$ is the regression function and $\|\cdot\|_\rho$ is the usual norm of $L^2(\mathcal X, d \rho_{\mathcal X}).$ In \cite{Smale2}, authors show that for any $0<\delta <1,$ the following holds with confidence level $1-\delta$:
\begin{equation}
\label{Error1}
\|\widehat f^\lambda_n-f_\rho\|_{\rho} \leq \frac{12 \kappa M \log\left(\frac{4}{\delta}\right)}{\sqrt{n \lambda}}+\lambda^r \| L_K^{-r} f_\rho\|_{\rho},
\end{equation} 
provided that 
\begin{equation*}
 \lambda \geq \frac{8 \kappa^2 \log\left(\frac{4}{\delta}\right)}{\sqrt{n}}.
\end{equation*}

In \cite{Vito}, the authors provide a procedure for choosing the optimal regularization parameter $\lambda$ (i.e. the value which yields the smallest unbiased risk 
for the estimator $\widehat f^{\lambda}_n$). More precisely, under the previous notations,  let  $F$ denote the closure of the range of $L_K.$ If $P_F$ denotes the orthogonal projection over $F$ and if  for some $0<r\leq 1,$ there exists a positive constant $C_r$ such  that  $\| L_K^{-r} P_F f\|_\rho \leq C_r,$ then for any $0<\delta <1,$ the following holds with probability at least $1-\delta$:
\begin{equation}\label{Rosasco_risk}
\mathcal R\big(\widehat f^{\lambda}_n\big) \leq \frac{M\kappa^2}{\lambda \sqrt{n}}\left(1+\frac{\kappa}{\sqrt{\lambda}}\right)\left(1+\sqrt{2 \log \frac{2}{\delta}}\right)+C_r \lambda^r,
\end{equation}
where $\mathcal R\big(\widehat f^{\lambda}_n\big)$ denotes the square root of the unbiased risk of $\widehat f^{\lambda}_n$. Hence, under the previous conditions, the optimal regularization parameter is asymptotically of order $O\big( n^{-\frac{1}{3+2r}}\big).$ Moreover, in \cite{Dieuleveut}, it has been shown that  under the additional hypothesis that the infinite sequence of eigenvalues $\mu_k$ of  $L_K$ has a decay rate of $O\big(k^{-\alpha}\big),$ the optimal regularization parameter is proportional to $O\big(n^{-\alpha/(2\alpha r+1)}\big).$ In this case, the convergence rate of the estimator $\widehat f^{\lambda}_n$ is given by 
\begin{equation}
\label{Dieuleveut_risk}
\| \widehat f^\lambda_n- P_F f\|_\rho= O\Big( n^{-\frac{\alpha r}{2\alpha r +1}}\Big).
\end{equation}
Other similar error bounds for $\widehat f^\lambda_n$ can be found in \cite{Hsu}. From \eqref{Estimatekappa2}, we note that each of the previous conditions on $\lambda$ also ensures numerical stability of the RKHS-Tikhonov  scheme \eqref{Estimator3}-\eqref{Tikhonov4}. We should also mention that when the closure of $L_K$ is dense in  $L^2(\mathcal X, d\rho_X),$ then $P_F=f.$ Also, the error analysis leading to \eqref{Dieuleveut_risk} is limited to a definite positive operator kernel with spectrum $(\lambda_k)_{k\geq 0}$ that decays at a polynomial rate, that is $\lambda_k=O(k^{-\alpha}),$  for some $\alpha >1.$ Unfortunately, this is not the case for both of our Legendre Christoffel-Darboux and Sinc kernels.  Nonetheless, in \cite{Hsu1} the authors have given error analyses and  procedures for the choice of the optimal regularizing parameter $\lambda$ and convergence rate under a fairly general condition on the decay rate of the spectrum of the operator $L_K.$ Such a general condition includes finite rank $L_K$ (which is the case for the Legendre kernel) or exponential decay rate of the spectrum of $L_K$ (which is the case of the Sinc-kernel) .

The error bounds \eqref{Error1}, \eqref{Rosasco_risk} and \eqref{Dieuleveut_risk} are all based on condition \eqref{Assumption1}.
In general, it is not trivial to check this condition, since it is based on some specific spectral properties of the operator $L_K$ and on the choice of the subspace of $L^2(\mathcal X, d \rho_X)$ where such a condition holds. In the special case where the Mercer's kernel is the Sinc kernel defined on $I^2=[-1,1]^2$, and $\rho_X$ is the uniform measure on $I,$ we check that if the true regression function $f$ is the restriction to $I$ of a $c-$bandlimited function $\widetilde f,$ then $f\in L_{K_c}^{-1/2}\big(L^2(I)\big).$  It is well known (see for example \cite{Bonami-Karoui3}) that   if $f$ is $c-$bandlimited, the double orthogonality of the eigenfunctions $\psi_{n,c}$ of $L_{K_c}$ over $I$ and $\mathbb R$ implies
$$ a_n(f) = <f,\psi_{n,c}>= \sqrt{\lambda_n(c)} <f,\psi_{n,c}>_{L^2(\mathbb R)}.$$
Consequently, we have 
$$ L_{K_c}^{-1/2} (f)=\sum_{n\geq 0} \lambda_{n}(c) ^{-1/2} <f,\psi_{n,c}>_{L^2(I)} \psi_{n,c}=\sum_{n\geq 0} <\widetilde f, \psi_{n,c}>_{L^2(\mathbb R)} \psi_{n,c}.$$
Thus, by Parseval's equality, one gets 
$$ \| L_{K_c}^{-1/2}(f)\|^2_{L^2(I)} = \sum_{n=0}^\infty |<\widetilde f,\psi_{n,c}>_{L^2(\mathbb R)}|^2 = \|\widetilde f\|^2_{L^2(\mathbb R)}<+\infty.$$
Also, for the Sinc kernel $K_c(\cdot,\cdot),$ we have $\kappa^2=\sup_{x\in I} K_c(x,x)=\frac{c}{\pi}.$ In order to minimize the $L^2(I)$ regression error,  we require, as in \cite{Smale2}, that for a given $0<\delta <1,$ the optimal theoretical admissible value $\lambda_{opt}$ of the regularization parameter 
$\lambda$ satisfies the inequality ${\lambda_{opt}= \frac{ 8 c  \log\left(\frac{4}{\delta}\right)}{\pi \sqrt{n}} <1.}$
Under this last condition and by  using \eqref{Error1} with $r=\frac{1}{2}$, we obtain the following lemma that provides an  $L^2(I)$-error analysis of the nonparametric regression RKHS-Tikhonov based scheme in  the special case of the Sinc kernel.

\begin{lemma}
Let $c$ be a positive real number and $0<\delta<1.$ Assume that the regression function $f$ is the restriction to $I=[-1,1]$ of a $c-$bandlimited function. For any integer $n\geq \left(\frac{8 c}{\pi}  \log\left(\frac{4}{\delta}\right)\right)^2,$ let $\widehat f^\lambda_{n,c}$ be given by \eqref{Estimator3}-\eqref{Tikhonov4} with $K=K_c.$ Then, with confidence level $1-\delta,$ we have :$$ \| \widehat f^\lambda_{n,c}- f\|_{L^2(I)} = O\left(\frac{ 8\sqrt{c}\log\left(\frac{4}{\delta}\right)}{\sqrt{\pi} n^{1/4}} \right).$$
\end{lemma}

\section{Numerical results}\label{numerical.study}

In this section, we give various  numerical tests that illustrate the nonparametric regression schemes of sections \ref{projection.operators} and \ref{Tikhonov}.

\noindent
{\bf Example 1:} In this example, we illustrate the Legendre and Sinc kernels projection estimators, as well as Tikhonov regularization scheme with Sinc kernel, for the approximation  of the regression function $f$ in model \eqref{Nregression1}. Here, $f$
is the restriction to $I$ of a bandlimited function with bandwidth $c=20.$ Thus, the model is
\begin{equation*}
Y_i = f(X_i)+\eta_i,\qquad f(x) =\frac{\sin(20 x)}{20 x},\quad x\in I,
\end{equation*}
where the random errors $\eta_i$ are distributed as $0.1 Z_i$, where $Z_i$ follows the standard normal distribution with mean zero and variance 1. We compute the two empirical projection estimators $\widehat f_{N,n}$ and $\widehat f_{c,n}$ with $N=20$ (for Legendre kernel) and $c=20$ (for the Sinc kernel). We consider the following sample sizes: $n=100, 500, 1000$. We also obtain the estimator $\widehat f^\lambda_n$ with the Sinc kernel $K_c(\cdot,\cdot)$ with $c=30$ and for $n=50, 100, 500$. For this last estimator, we use generalized cross validation (GCV) to choose the appropriate value $ \lambda_{GCV}$ of $\lambda$, that is, the value which provides an error close to the minimum regression error. Note that under the notations of section \ref{Tikhonov},  $\lambda_{GCV}$ is given by 
$$ \lambda_{GCV}=\arg\min \frac{\sum_{i=1}^n \big(\widehat f^\lambda_n(X_i)- Y_i\big)^2}{n \left(1 - \frac{\mbox{Trace} (G_\lambda^{-1} G_0)}{n}\right)^2}.$$
Here, we obtain $\lambda_{GCV}= 0.01.$ For illustration, in Figure \ref{figure1}, we plot the true function $f$, its noised version  $\widetilde f,$ and the approximation  of $f$ by the empirical projection estimator  $\widehat f_{c,n}$  with $n=1000$ and $c=20.$  We use a highly accurate Gaussian quadrature scheme to compute the $L^2(I)$-regression errors $\|f-\widehat f_{N,n}\|,$  $\|f-\widehat f_{c,n}\|$ and $\|f-\widehat f^{\lambda}_{n}\|$.
This simulation scenario is repeated 500 times and the average regression errors  are given in Table \ref{table1}. The numerical results indicate that the three estimators $\widehat f_{c,n},$ $\widehat f_{N,n}$ and $\widehat f^\lambda_{n}$ have very similar precision, for the respective values of their parameters. We also compute the average running time of each estimator. The estimator $\widehat f_{c,n}$ is the fastest one. For example, we found that computing $\widehat f_{c,n}$ with $n=1000$ is $15$ times faster than computing $\widehat f^\lambda_n$ with $n=150.$  This is coherent with the well known fact (see for example \cite{Yang}) that the time complexity for computing the estimator $\widehat f^\lambda_n$ (also called kernel ridge regression estimator) is $O(n^3).$ On the other hand, the empirical projection estimator, which requires a vector-matrix multiplication, has an  $O(n^2)$ complexity.

\begin{center}
\begin{table}[h]
\vskip 0.02cm\hspace*{4cm}
\begin{tabular}{ccccc} \hline
 $n$ &$\|f-\widehat f_{c,n} \|$&$\|f-\widehat f_{N,n} \|$& $n$  & $\|f-\widehat f^\lambda_{n} \|$ \\   \hline
$100$  &$1.64\mbox{e}-01$  & $1.32\mbox e-01$ &$50$  &$1.37\mbox e-01$ \\
$500$  &$6.48\mbox e-02$  & $6.07\mbox e-02$ &$100$ &$7.39\mbox e-02$ \\
$1000$  &$4.78\mbox e-02$  & $3.49\mbox e-02$&$150$ &$5.67\mbox e-02$  \\
\hline
\end{tabular}
\caption{$L^2(I)-$regression errors corresponding to example 1.}
\label{table1}
\end{table}
\end{center}
\begin{figure}[h]\hspace*{1cm}
 {\includegraphics[width=15.05cm,height=4.5cm]{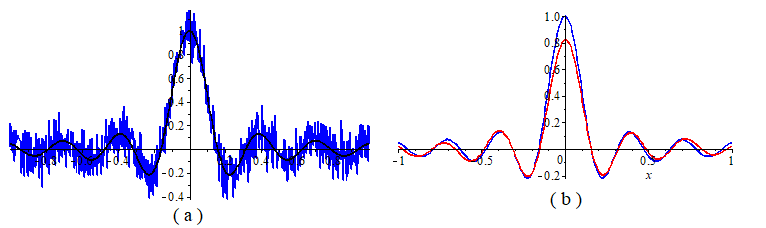}}
 \vskip -0.5cm\hspace*{1cm} \caption{(a) Plots of $f$ (black) and $\widetilde f$ (blue) $\quad$ (b) Plots of $f$ (blue) and its approximation
 $\widehat f_{c,n}$ (red)}
\label{figure1}
 \end{figure} 
 
 \noindent {\bf Example 2:} In this example, we illustrate the accuracy of our three estimators $\widehat f_{c,n},$ $\widehat f_{N,n}$ and $\widehat f^\lambda_{n}$ in the case of a regression function belonging to the Sobolev space $H^s(I).$ We consider for $f$ the Brownian motion  function  $g^s(x)$ given by:
 \begin{equation}\label{eq5.1}
 g^s(x)= \sum_{k\geq 1} \frac{X_k}{k^s} \cos(k\pi x),\quad -1\leq  x\leq 1,
 \end{equation}
 where $s$ is a positive real number and the $X_k$'s are standard Gaussian random variables. It is well known that  $g^s \in H^s(I).$ The random noise terms $\eta_i$ are taken as $\eta_i=0.1 Z_i$, with $Z_i$ distributed as a standard normal, as in example 1. We consider two values for $s$, namely $s=1, 2.$ Then, we calculate $\widehat f_{c,n},$ $\widehat f_{N,n}$ and $\widehat f^\lambda_{n},$  with $c=30,$ $N=20,$ $n=100, 500, 1000$ for the first two estimators and $n=50, 100, 150$ for  $\widehat f^\lambda_{n}.$ For this last estimator, the regularization parameter, chosen by  GCV, is equal to 0.01.
We simulate 500 samples and we obtain the average (over the 500 samples) $L^2(I)$-regression error for each estimator. These results are reported in Table \ref{table2}. In Figure \ref{figure2}, we plot the graphs of $g^s$ and of its noised version  $\widetilde g^s,$ for $s=1.$ We also plot  the estimator $\widehat g^s_{c,n}$ with $n=1000$ and $c=30.$
    \begin{center}
 \begin{table}[h]
 \vskip 0.2cm\hspace*{3cm}
 \begin{tabular}{cccccc} \hline
    $s$ &$n$&$\| g^s - \widehat g^s_{c,n}\|$&$\| g^s - \widehat g^s_{N,n}\|$&$n$& $\| g^s - \widehat g^{s,\lambda}_n\|$  \\   \hline
   $1$  &$100$ &$4.91\mbox e-01$ &$4.85\mbox e-01$ & $50$ & $4.72 \mbox e-01$     \\
          &$500$ &$3.66\mbox e-01$ &$3.60\mbox e-01$ & $100$& $3.13 \mbox e-01$     \\
          &$1000$&$3.73\mbox e-01$ &$3.47\mbox e-01$ & $150$& $2.88 \mbox e-01$     \\
    $2$ &$100$ &$3.03\mbox e-01$ &$2.90\mbox e-01$ & $50$ & $1.64 \mbox e-01$     \\
          &$500$ &$1.38\mbox e-01$ &$1.30\mbox e-01$ & $100$& $8.34 \mbox e-02$     \\
          &$1000$&$9.80\mbox e-02$ &$9.18\mbox e-02$ & $150$& $6.48 \mbox e-02$     \\
  \hline
  \end{tabular}
  \caption{$L^2(I)-$regression errors corresponding to example 2.}
 \label{table2}
 \end{table}
 \end{center}
 
\begin{figure}[h]\hspace*{1cm}
 {\includegraphics[width=15.05cm,height=5.0cm]{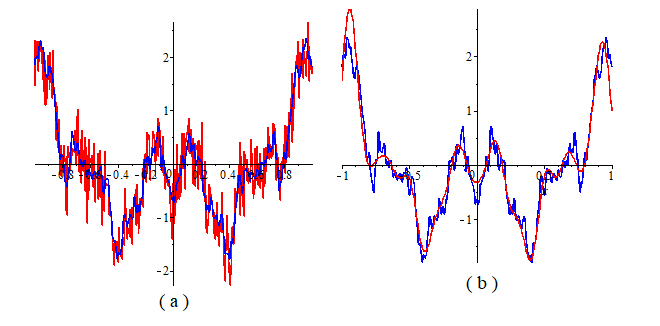}}
 \vskip -0.5cm\hspace*{1cm} \caption{(a) Plots of $g^s$ (blue) and $\widetilde g^s$ (red) $\quad$ (b) Plots of $g^s$ (blue) and its approximation  $\widehat g^s_{c,n}$ (red)}
\label{figure2}
 \end{figure}

\noindent
{\bf Example 3:} In this  example, we illustrate the estimates (given in Proposition \ref{2-condtion.prop}) of the $2$-condition numbers of the  Gram matrices. For this purpose, we consider the Gram random matrix $G^c_{0,n}$ and its regularized version $G^c_{\lambda,n}$ given by \eqref{GramMatrices}, with $\lambda= 1 \mbox e-04$ and various  values for $c$ and $n$ ($c=30, \, 50$ and $n=50, 75$). We compute highly accurate values of the mean over $10$ realizations, for the $2$-condition numbers $\kappa_2(G^c_{0,n})$ and $\kappa_2(G^c_{\lambda,n})$. We also calculate the theoretical upper bound of $\kappa_2(G^c_{\lambda,n})$, given by \eqref{Estimate2}. Numerical results are given in Table \ref{tableau3} and are consistent with the theoretical results stated in Proposition \ref{2-condtion.prop}.
\begin{center}
\begin{table}[h]
\vskip 0.2cm\hspace*{3cm}
\begin{tabular}{ccccc} \hline
 $c$  &$n$&$\kappa_2(G^c_{0,n})$&$\kappa_2(G^c_{\lambda,n})$&$1+\frac{1}{\lambda}\left(1+\frac{c}{\pi \sqrt{n}}\right)$ \\   \hline
$30$  &$50$     &$1.081148\mbox e+60$  & $1.104737\mbox e+04$ & $2.350574\mbox e+04$    \\
      &$75$     &$2.046118\mbox e+117$ & $1.020951\mbox e+04$ & $2.102758\mbox e+04$  \\
$50$  &$50$     &$2.086150\mbox e+39$  & $1.568333\mbox e+04$ & $3.250891\mbox e+04$    \\
      &$75$     &$1.059255\mbox e+89$ &  $1.344763\mbox e+04$ & $2.837863\mbox e+04$  \\
\hline
\end{tabular}
\caption{Numerical results for example 3.}
\label{tableau3}
\end{table}
\end{center}

\end{document}